\documentclass{amsart}
\usepackage[utf8]{inputenc}
\usepackage{comment}
\usepackage{amsmath}
\usepackage{mathtools}
\usepackage{}
\usepackage{graphicx}
\usepackage[colorlinks=true, allcolors=blue]{hyperref}
\usepackage{amsfonts}
\usepackage{amsthm}
\usepackage{newlfont}
\usepackage{amscd}
\usepackage{amsgen}
\usepackage{amssymb}
\usepackage{mathrsfs}	
\usepackage{longtable}
\usepackage{listings}
\usepackage{extarrows}
\usepackage{tikz}
\usepackage{tikz-cd}
\usepackage{verbatim}
\numberwithin{equation}{section}
\usepackage[all]{xy}
\usepackage{color}
\usepackage{amssymb}
\usepackage[left=3cm,right=3cm]{geometry}
\usepackage{tikz}
\usepackage{tikz-cd}
\usepackage{mathtools}
\usepackage{bm} 
\usetikzlibrary{matrix,shapes,arrows,decorations.pathmorphing}
\usepackage{calligra}
\usepackage{mathrsfs}
\usepackage{enumitem} 
\usepackage{float}
\restylefloat{table}
\usepackage{ytableau}
\usepackage{adjustbox}
\usetikzlibrary{calc,matrix,positioning}

\tikzset{ 
    table/.style={
        matrix of nodes,
        row sep=-\pgflinewidth,
        column sep=-\pgflinewidth,
        nodes={rectangle, text width=2.5em, text height = 1.5em, align=center},
        text depth=1.25ex,
        text height=2.5ex,
        nodes in empty cells
    },
}

\usepackage{comment}
\usepackage{todonotes}

\newcommand{\kieran}[1]{\todo[inline, size=\tiny, author=Kieran, backgroundcolor=blue!20]{#1}}

\let\blb\mathbb

\def\CC{{\blb C}}

\def\PP{{\blb P}}

\DeclareMathOperator{\Gr}{Gr}

\newcommand{\mQ}{\mathcal{Q}}
\newcommand{\mU}{\mathcal{U}}

\newcommand{\mE}{\mathcal{E}}
\newcommand{\mA}{\mathcal{A}}
\newcommand{\mK}{\mathcal{K}}

\DeclareMathOperator{\Sym}{Sym}
\newcommand{\of}{\mathcal{O}}
\newcommand{\W}{\bigwedge}

\newcommand{\mF}{\mathcal{F}}

\newtheorem{lemma}{Lemma}[section]
\newtheorem{proposition}[lemma]{Proposition}
\newtheorem{theorem}[lemma]{Theorem}
\newtheorem{corollary}[lemma]{Corollary}

\theoremstyle{remark}

\newtheorem{remark}[lemma]{Remark}

\def\End{\operatorname{End}}

\def\Ext{\operatorname{Ext}}

\def\ch{\operatorname{ch}}

\DeclareMathOperator{\sHom}{\mathscr{H}\text{\kern -3pt {\calligra\large om}}\,}

\newcommand\quotient[2]{
        \mathchoice
            {
                \text{\raise1ex\hbox{$#1$}\Big/\lower1ex\hbox{$#2$}}%
            }
            {
                #1\,/\,#2
            }
            {
                #1\,/\,#2
            }
            {
                #1\,/\,#2
            }
    }

\newcounter{gensurf}
\title{Examples of non-rigid, modular vector bundles on hyperk\"ahler manifolds}
\author{Enrico Fatighenti}
\address{\newline
Alma Mater studiorum Universit\`a di Bologna\hfill\newline
Dipartimento di Matematica\hfill\newline
Piazza di porta San Donato 5, 40126 Bologna, Italy}
\email[E.~Fatighenti]{enrico.fatighenti@unibo.it}

\begin{document}
\maketitle
\begin{abstract}
    We exhibit examples of slope-stable and modular vector bundles on a hyperk\"ahler manifold of K3$^{[2]}$-type which move in a 20-dimensional family and study their algebraic properties. These are obtained by performing standard linear algebra constructions on the examples studied by O'Grady of (rigid) modular bundles on the Fano varieties of lines of a general cubic 4-fold and the Debarre-Voisin hyperk\"ahler manifold. 
\end{abstract}

\section{Introduction}

The purpose of this short note is to produce examples of slope-stable modular vector bundles on hyperk\"ahler manifolds, that are not rigid. These objects are of interest since it is expected that they could lead to the construction of new explicit (possibly singular) examples of families of polarized hyperk\"ahler manifolds. In particular, one would like to generalize the theory of moduli spaces of sheaves on K3 and abelian surfaces, from which all the current known deformation types of hyperk\"ahler manifolds are constructed. The idea is to check if it is possible to take a suitable moduli space of modular sheaves on a higher-dimensional hyperk\"ahler manifold, and obtain in this way another hyperk\"ahler manifold. This idea has its roots in the work of Verbitsky, \cite{ver}, but it was properly formalized by O'Grady, \cite{og19}, and the theory has recently witnessed a large inflationary phase, see e.g. the works of \cite{be22,bot22,mo23, markman, og22}. 

The modularity condition for hyperk\"ahler (HK) manifolds of type K3$^{[2]}$ amounts to showing that the discriminant of a sheaf, see equation \ref{discr}, is a multiple of the second Chern class of the variety.
In \cite{og19}, O'Grady studied examples of such bundles on certain hyperk\"ahler varieties, (in this case also projectively hyperholomorphic, i.e. such that the projectivization of the bundle extends to all deformations of the base variety) which however were rigid. Very recently, in \cite{bot22} Bottini used a non-rigid vector bundle with 10-dimensional deformation space to construct a component of a moduli space equipped with a closed holomorphic two-form which is birational (and conjecturally isomorphic) to a hyperk\"ahler of OG10 type.

It, therefore, becomes interesting to produce examples of these types of vector bundles on hyperk\"ahler manifolds. This is the scope of this paper.
To be precise, we will show that starting from the modular slope-stable rigid bundles studied by O'Grady, and performing only standard linear algebra constructions, we can produce other modular slope-stable bundles which are not rigid. We suspect that this in fact should be the shadow of a more general phenomenon, and a recipe to produce many more bundles of this sort.

We proceed as follows: starting from the Fano variety $X$ of lines on a general cubic 4-fold, we consider $\mE$, the restricted rank 4 quotient bundle from the Grassmannian of two dimensional subspaces on a 6-dimensional vector space. This is slope-stable, modular but rigid, see \cite{og19}. We perform simple representation-theoretic constructions on $\mE$, obtaining two other bundles which we show to be still slope-stable and modular, but with a 20-dimensional first cohomology group of their endomorphism bundles. For the second of the two bundles, we also show that the traceless part of the self-$ext^2$ space is isomorphic to the second exterior power of $ext^1$, which amounts to the unobstructedness of the bundle itself, as in \cite[Thm 6.1]{be22} and \cite{mo23}. In short, our main results can be summarized as follows:
\begin{theorem}[See Thm \ref{thm:ext},\ref{recap1},\ref{thm:ext2}, \ref{prop:stability2}] \label{thm1.1}
Let $X \subset \Gr(2,6)$ be the Fano variety of lines on a general cubic fourfold. There exist two bundles $\mF$, $\mK$ on $X$, of rank (respectively) 6 and 15 which are slope-stable and modular. Their $ext$-cohomology dimension table is
\[ ext^p(\mF, \mF) \cong
    \begin{cases}
         20 \ (p=1,3) \\
        2 \ \ (p=2) \\
        1 \ \ (p=0,4) \\
        0 \ \ \textrm{otherwise}.
    \end{cases}\] and

    \[ ext^p(\mK, \mK) \cong
    \begin{cases}
         20 \ (p=1,3) \\
        191 \ \ (p=2) \\
        1 \ \ (p=0,4) \\
        0 \ \ \textrm{otherwise}.
    \end{cases}\]
    
Also, in the second case, if we consider the traceless part of the endomorphism bundle, we have the isomorphism
$$\W^2 \Ext^1(\mK, \mK) \cong H^2(End^0(\mK)).$$
\end{theorem}

In the above theorem, the traceless part of the endomorphism bundle we denote the kernel of the natural trace map $End(\mF) \to \of_X$. In other words,  $\mF \otimes \mF^{\vee} \cong End^0(\mF) \oplus \of_X$, and similarly for $\mK$.

We are able to perform the same computation starting from the Debarre-Voisin hyperk\"ahler fourfold $Z \subset \Gr(6,10)$, and obtain exactly the same results. This is very likely not to be a coincidence, since the projectivization of these vector bundles should deform one into the other, see remark \ref{rmk:dv}.

We briefly explain the rationale behind performing these computations. It is a standard fact (see for example \cite[Section 4.5]{hl10}, that the tangent space to the moduli space at, say, $\mF$ is given by $\mathrm{Ext}^1(\mF, \mF)$, which in this case is equal to its traceless part. The obstructions space lies in the traceless part of $\mathrm{Ext}^2(\mF, \mF)$, i.e. $H^2(End^0(\mF))$. Thanks to the computations above, one has that if the symmetrized Yoneda pairing
\[
H^1(X, End^0(\mF)) \times H^1(X, End^0(\mF)) \to H^2(End^0(\mF))
\]
is zero, then there is a single component of the moduli space of stable sheaves on $X$ containing $\mF$, it has dimension 20 and it is smooth at $\mF$. This example (and others) are currently studied by O'Grady in a forthcoming work, in connection with a construction of locally complete, polarized families of $K3^{[10]}$. 
The second example $\mK$ is studied for similar reasons. In fact, even though in this case $ext^2$ is quite large, the last isomorphism in Theorem \ref{thm1.1}, together with recent formality results in \cite{mo23}, imply the smoothness of the Kuranishi space. 
 
The techniques used in this paper are representation-theoretic, and could be easily adapted to produce other examples. In fact, we expect that using similar tricks one should be able to produce many more examples of such bundles, and (possibly) produce explicit constructions of hyperk\"ahler manifolds.

\subsection*{Acknowledgments}
The present paper grew out of a question asked by Kieran O'Grady. I wish to thank him for this and the many discussions. I also thank in particular Francesco Meazzini and Claudio Onorati (who also introduced me to this notion), and also Simone Billi, Francesco Denisi, Lucas Li Bassi, Giovanni Mongardi, Arvid Perego and Andrea Petracci. The author was partially supported by PRIN2020 2020KKWT53, and is a member of INDAM-GNSAGA.

\subsection*{Notation}
With $X= (G, \mF)$ we mean that $X= V(s) \subset G$, for a general global section $s \in H^0(G, \mF)$. With $\Gr(k,n)$ we mean the Grassmannian of $k$-dimensional subspace in a $n$-dimensional complex vector space. We denote with $\mU$ its rank $k$ tautological bundle, $\mQ$ its rank $n-k$ quotient bundle. $V_n$ denotes a $n$-dimensional complex vector space. All varieties are assumed to be smooth and projective.
\section{The computation}

Set $X= (\Gr(2,6), \Sym^3 \mU^{\vee})$, i.e. $X$ is a hyperk\"ahler manifold of K3$^{[2]}$-type, in fact identified with the variety of lines on a generic cubic hypersurface $Y \subset \PP^5$, polarized by the Pl\"ucker line bundle. On $X$ there is a rank 4 vector bundle $\mE$, which is the restriction $\mQ|_X$, where $\mQ$ is the rank 4 quotient bundle on the Grassmannian $\Gr(2,6)$.

In \cite[Thm. 1.4]{og19} it is proved that this vector bundle $\mE$ is slope-stable, modular and rigid, meaning that $H^p(X, End^0(\mE))=0 $ for all $p$. In fact, one can prove directly that $\mE \otimes \mE^{\vee} \cong End^0(\mE) \oplus \of_X$, and one can check the acyclicity of the traceless part.

The above result led to the quest for examples of vector bundles on HK which are still modular and slope-stable, but not rigid. In fact, we are going to prove that if we perform simple representation-theoretic constructions on $\mE$, the resulting vector bundle is almost never rigid. We quote our first result here.
\begin{theorem} \label{thm:ext}
    Set $\mF =\W^2 \mE$ on $X$, using the above notation. Then $\mF$ is modular, but non-rigid. In particular, if we define $End^0(\mF)$ as $(\mF \otimes \mF^{\vee}) / \of_X$, we have \[ H^p(X, End^0(\mF)) \cong
    \begin{cases}
         20 \ (p=1,3) \\
        1 \ \ (p=2) \\
        0 \ \ \textrm{otherwise}.
    \end{cases}\]
\end{theorem}

\begin{remark} Since $\W^2\mQ$ is irreducible, then is slope-stable \cite{ram66,um78}. However, one has to prove the slope-stability of the restriction, while the simpleness that can be checked directly. In fact $\mF \otimes \mF^{\vee}$ decomposes as the direct sum of $End^0(\mF) \oplus \of_X$. In particular, the full ext table of $\mF$ is
\[ ext^p(X, \mF) \cong
    \begin{cases}
         20 \ (p=1,3) \\
        2 \ \ (p=2) \\
        1 \ \ (p=0,4) \\
        0 \ \ \textrm{otherwise}.
    \end{cases}\]

This agrees with the Euler characteristic computation, which can be done using for example Riemann-Roch, and that shows $\chi(End(\mF))=-36$.
\end{remark}

Before checking the cohomology computation, we want to show that $\mF$ is modular. In order to do this, we use the following Proposition \ref{lem:modularlambda}. We first recall a definition. Given a rank $r$, torsion-free sheaf $F$ on a manifold $Z$, the discriminant $\Delta(F)$ of $F$ is defined as
\begin{equation}\label{discr}
    \Delta(F)=2rc_2(F)-(r-1)c_1(F)^2= -2r \ch_2(F) +\ch_1^2(F).
\end{equation}

As in \cite[Definition 1.1]{og19} a torsion-free coherent sheaf $F$ on a hyperk\"ahler of dimension $2n$ is said to be modular if there exists a rational number $d(F)$ such that for all $\alpha \in H^2(X)$ one has 
\[
\int_X \Delta(F) \cdot \alpha^{2n-2}=d(F) \cdot (2n-3)!! \cdot q_X(\alpha)^{n-1},
\]
with $q_X$ the Beauville-Bogomolov-Fujiki quadratic form.

Following \cite[Remark 1.3 and Section 2.1]{og19}, it follows that on a hyperk\"ahler manifold of $K3^{[2]}$-type $F$ is modular if and only if $\Delta(F)$ is a multiple of $c_2(X)$. Since this holds for $\mE$ on $X$, then $\mE$ is modular. The following Proposition shows that the same holds for $\mF$.

\begin{proposition} \label{lem:modularlambda}
Let $F$ be a torsion-free coherent sheaf of rank $r$ on a smooth projective variety $Z$. Then
\[
\Delta(\W^pF)= \lambda_p \Delta(F),
\]
where 
\[
\lambda_p:=\frac{1}{p-1} {r-1 \choose p} {r-2 \choose p-2}
\]
In particular, $\mF$ on $X$ defined as above is modular. 
\end{proposition}
In order to prove the above proposition, we will use the following auxiliary Lemma, which we highlight separately since it could be useful for further computations:

\begin{lemma}\label{lem:chernwedge}
    Let $F$ be as above. The following formulae hold:
    \begin{itemize}
        \item $\ch_1(\W^p F)= {r-1 \choose p-1}\ch_1(F)$; 
        \item $\ch_2(\W^pF) = \frac{1}{2} {r-2 \choose p-2} \ch_1^2(F) + {r-2 \choose p-1 }\ch_2(F)$.
    \end{itemize}
\end{lemma}
\begin{proof}
The first part follows immediately from the splitting principle, as we are going to show. Similar formulae, although less explicit, can be also extracted from \cite{dra}. If we write formally $F$ as $\bigoplus_{i=1}^r L_i$, we have $\W^p F= \bigoplus_{i_1 < \ldots < i_p} L_{i_1} \otimes \ldots\otimes L_{i_p}$, hence taking the Chern character, \begin{equation} \label{chernsplitting}   
\ch(\W^pF)= \sum_{i_1 < \ldots < i_p}\ch(L_{i_1}) \cdot \ldots \cdot \ch(L_{i_p}) 
\end{equation}
If we extract the degree 1 component, and denote with $x_i=c_1(L_i)$, we have that every $x_i$ appears in the sum exactly ${r-1 \choose p-1}$ times, i.e.
\[
\ch_1(\W^p F)= {r-1 \choose p-1}(x_1 +\ldots +x_r)= {r-1 \choose p-1} \ch_1(F).
\]

For the second formula, we extract the second degree component from equation \ref{chernsplitting}. We have
\[
\ch_2(\W^pF)= \sum_{i_1 < \ldots < i_p} \left[ \left(\frac{x_{i_1}^2}{2} + \ldots + \frac{x_{i_p}^2}{2}\right) + \sum_{1 \leq h <k\leq p} x_{i_h}x_{i_k}\right]
\]
We can rewrite this as
\begin{equation} \label{eqn2splitting}
    \ch_2(\W^pF)={r-1 \choose p-1}\ch_2(F) + {r-2 \choose p-2} \sum_{i<j} x_ix_j.
\end{equation}
Notice that the last term is simply $c_2(F)$ up to a multiple. Hence, we can use the defining relation for the second Chern character in terms of the Chern classes, i.e. 
\[
\frac{1}{2} {r-2 \choose p-2} \ch_1^2(F) -{r-2 \choose p-2}\ch_2(F)= {r-2 \choose p-2} \sum_{i<j}x_ix_j.
\]
We can therefore substitute this in equation \ref{eqn2splitting} and rewrite
\[ \ch_2(\W^pF)=\frac{1}{2} {r-2 \choose p-2} \ch_1^2(F) + {r-2 \choose p-1 }\ch_2(F),\]
which concludes the proof.

\end{proof}
We now use Lemma \ref{lem:chernwedge} to prove Proposition \ref{lem:modularlambda}.

\begin{proof}[Proof of Proposition \ref{lem:modularlambda}]
We expand the definition of $\Delta(\W^pF) =\ch_1^2(\W^pF)-2 {r \choose p} \ch_2 (\W^pF)$. Using Lemma \ref{lem:chernwedge}, we obtain
\[
\Delta(\W^p F) = {r-1 \choose p-1}^2 \ch_1^2(F) -2 {r \choose p}\left[\frac{1}{2} {r-2 \choose p-2} \ch_1^2(F) + {r-2 \choose p-1 }\ch_2(F)\right].
\]
We can therefore expand the right-hand side of the above equation as
\[
\frac{(r-1)^2(r-2)^2 \ldots (r-p+1)^2}{(p-1)^2 (p-2)!^2} \ch_1^2(F)+\]\[ -2 \frac{r (r-1)(r-2) \ldots (r-p+1)}{p(p-1)(p-2)!} \left[ \frac{(r-2) \ldots (r-p+1)}{2 (p-2)!}\ch_1^2(F) + \frac{(r-2) \ldots (r-p+1)(r-p)}{(p-1)(p-2)!}\ch_2(F) \right]
\]
which further simplifies to
\[
\frac{(r-1)(r-2)^2 \ldots (r-p+1)^2}{(p-1)(p-2)!^2} \left[ \left( \frac{r-1}{p-1}- \frac{r}{p} \right) \ch_1^2(F)  - 2 \frac{r(r-p)}{p(p-1)} \ch_2(F)\right],
\]
which again is equal to
\[
\frac{1}{(p-1)} {r-1 \choose p} {r-2 \choose p-2} \left[ \ch_1^2(F)-2r \ch_2(F)\right]= \frac{1}{p-1} {r-1 \choose p} {r-2 \choose p-2} \Delta(F),
\]
which concludes the proof.    
\end{proof}

\begin{remark}
    We are inclined to believe that a similar formula should hold for any Schur power $\Sigma_{\alpha}F$, but extracting a general formula using these techniques might be quite tricky. One good exercise could be to extract the coefficient for the symmetric power $\Sym^p(F)$. 
    For example, for the second symmetric power, one has that
    \[
\ch_1(\Sym^2 F) =(r+1) \ch_1(F); \ \ \ch_2(\Sym^2 F)= \frac{1}{2}\ch_1^2(F)+(r+2) \ch_2(F);
    \]
    hence $\Delta(\Sym^2 F)= {r+2 \choose 2} \Delta(F)$.
    
    Notice that in the locally free case the result for the second symmetric power is immediately implied by the above result, together with the fact that the tensor product of modular sheaf is modular, see \cite[Remark 2.1]{og19}.
\end{remark}

Using the same techniques, we give a small corollary which could be useful in future computations. In fact, while in general the discriminant behaves badly with respect to the direct sum, we can obtain some positive results in some special cases.

\begin{lemma} Let $E$ be a torsion-free coherent sheaf of rank $r$ on a smooth projective variety $Z$. One has
  $\Delta(E^{\oplus n})= n^2 \Delta(E)$.
\end{lemma}
\begin{proof}
    It suffices to expand the definition of $\Delta$. We do the $n=2$ case first. Consider first $E \oplus F$: we have
    \[
    \Delta(E \oplus F) = (\ch_1(E) + \ch_1(F))^2-2(r_E+ r_F) (\ch_2(E)+\ch_2(F))=\]\[= \Delta(E) + \Delta(F)+2 \ch_1(E)\ch_1(F)- 2r_F \ch_2(E)- 2r_E \ch_2(F).
    \]
    If now we take $E=F$, we have
\[
\Delta(E^{\oplus 2})= 2 \Delta(E) +2(\ch_1^2(E)-2r \ch_2(E))=4 \Delta(E).
\]

In general we may proceed by induction: we compute in fact
\[\Delta(E \oplus E^{\oplus n})= \ch_1^2(E \oplus E^{\oplus n})-2(n+1)r \ch_2(E \oplus E^{\oplus n})\]
We expand this sum and we use additivity of the Chern character to obtain
\[
\left[ \ch_1^2(E^{\oplus n})-2nr \ch_2(E^{\oplus n})\right]+ \left[ \ch_1^2(E)-2r \ch_2(E)\right] + 2n \ch_1^2(E) -4nr \ch_2(E).
\]
By inductive hypothesis, the latter is equal to
\[
\Delta(E) (n^2+1+2n)= (n+1)^2 \Delta(E),
\]
which concludes the proof.
\end{proof}

We are now in position to prove Theorem \ref{thm:ext}. Thanks to the above Propositions \ref{lem:modularlambda}, we will only need to check the cohomological computation and the stability. 
Before this, we decompose $\mF \otimes \mF^\vee$ in irreducible summands, using Littlewood-Richardson rule, see e.g. \cite{wey03}. We have in fact:
\begin{equation} \label{decomposition}
    \mF \otimes \mF^\vee \cong \Sigma_{2,2}\mQ(-1)|_X \oplus \Sigma_{2,1,1}\mQ (-1)|_X \oplus \of_X,
\end{equation}
where with $\Sigma_\alpha \mQ$ we denote the Schur functor associated to the partition $\alpha$ applied to $\mQ$.
 The first two factors are therefore the irreducible summands of $End^0(\mF)$.
Although the decomposition given in equation \ref{decomposition} might seem a bit obscure at a first sight, it is in fact quite natural. In fact, $\Sigma_{2,1,1}\mQ$ is nothing but $\W^2 \W^2 \mQ$, and of course $\Sigma_{2,2}\mQ|_X \oplus \of_X$ is nothing but $\Sym^2 \W^2 \mQ|_X$. Therefore, we can rewrite the decomposition in equation \ref{decomposition} as 
\begin{equation} \label{symwedgedecomposition}
    \mF \otimes \mF^\vee \cong (\Sym^2\mF \oplus \W^2 \mF)_0 \otimes det(\mU) \oplus \of_X,
\end{equation} where the first term on the right hand side denotes the traceless part of the symmetric and skew-symmetric components of the endomorphism bundle.

\begin{remark}
   Notice that all the summands of the decomposition in equation \ref{decomposition} are self-dual, in the sense that $G_i \cong G_i^{\vee}$. This can be checked directly at the level of the ambient Grassmannian, using the classical rules of duality for Schur functors, see \cite[Section 2.5]{Ku95} The same phenomenon happens for $\mE$, whose endomorphism bundle decomposes as $\Sigma_{2,1,1}\mQ(-1)|_X \oplus \of_X$, both of which are self-dual.
\end{remark}

\begin{proof}[Proof of Theorem \ref{thm:ext}]
    The main tools for this proof are the Koszul complex for $X$ in $\Gr(2,6)$ and the Borel--Bott--Weil theorem, for computing the cohomology groups of homogeneous vector bundles on homogeneous varieties. For a general introduction we refer to \cite[4.1.4]{wey03}, and \cite{dft} for examples of similar computations. The strategy goes as follows: the normal bundle of $X$ in $\Gr(2,6)$ is the restriction to $X$ of $\Sym^3 \mU^\vee$: hence the Koszul complex for $X$ is 
    \begin{equation} \label{koszul}
        0 \to det (\Sym^3 \mU) \to \W^3 (\Sym^3 \mU) \to \W^2 (\Sym^3 \mU) \to \Sym^3 \mU \to \of_{\Gr(2,6)} \to \of_X \to 0
    \end{equation}

However, the factors in the above complex are not all irreducible, and we need to compute their decomposition using Littlewood-Richardson, namely:
\begin{equation} \label{koszulLR}
     0 \to \of(-6) \to  (\Sym^3 \mU) (-3)  \to (\Sym^4 \mU)(-1) \oplus \of(-3) \to \Sym^3 \mU \to \of_{\Gr(2,6)} \to \of_X \to 0.
\end{equation}

In fact, one can derive the above decomposition using duality to compute the third exterior power, and formula for tensor products of representations and the second symmetric power of a symmetric power, see for instance \cite[Corollary 2.3.5]{wey03}, \cite[Example 1.8.9]{mac}.
    
    Hence, if we want to compute the cohomology groups of $End^0(\mF)$ it will suffice to tensor the above complex \ref{koszulLR} with the first two summands of the decomposition of $End(\W^2 \mQ)$. The proof of this theorem is quite long and involved: hence, we will divide it in several steps. 
    To simplify the notation, we will call the vector bundles
    \[
\mA_1:= \Sigma_{2,2}\mQ(-1);
    \]
        \[
\mA_2:= \Sigma_{2,1,1}\mQ(-1);
    \]
    such that $End^0(\mF) \cong \mA_1|_X \oplus \mA_2 |_X$.

\subsubsection*{Step 1: $\mA_1$ and $\mA_2$ are acyclic}
    This is a simple application of the Borel-Bott-Weil (BBW) rule on $\Gr(2,6)$. To be precise, to $\mA_1$ we associate the partition $a_1=(2,2,0,0,1,1)$ and to  $a_2=\mA_2$ we associate the partition $(2,1,1,0,1,1)$, where we are using the convention that to an irreducible homogeneous bundle $\Sigma_\alpha \mQ \otimes \Sigma_\beta \mU$ is associated the partition $(\alpha, \beta)$. Let $\delta$ be $\delta=(5,4,3,2,1,0)$. Then by the BBW theorem, $\mA_i$ will be acyclic if the vectors $a_i+\delta$ contain repeated entries. This is the case for both factors. This means that the non-zero cohomology appearing in Theorem  \ref{thm:ext} will come from a necessary and detailed analysis of the cohomologies from complex \ref{koszulLR}.

\subsubsection*{Step 2: $\mA_2|_X$ is acyclic}
 In order to compute the cohomology of $\mA_2|_X$ we first need to tensor complex \ref{koszulLR} with $\mA_2$.

We get the following complex, after tensoring with $\mA_2$:

\begin{equation}\label{kosza1}
 0 \to \Sigma_{2,1,1}\mQ (-7) \to  \Sigma_{2,1,1}\mQ \otimes (\Sym^3 \mU)(-4) \to  \Sigma_{2,1,1}\mQ \otimes ((\Sym^4 \mU)(-2)  \oplus \of(-4) ) \to\end{equation}
\[\to   \Sigma_{2,1,1}\mQ \otimes (\Sym^3 \mU)(-1) \to \mA_2 \to \mA_2 |_X \to 0.
\]

 Since all of these bundles are of the form $\Sigma_{2,1,1}\mQ \otimes \Sigma_{\beta_1, \beta_2} \mU$, any factor with $\beta_1 \in \lbrace 6,4,3,1 \rbrace$ or $\beta_2 \in \lbrace 7,5,4,2 \rbrace$ will be acyclic, since in these cases $(2,1,1,0, \beta_1, \beta_2)+\delta$ will have repeated entries. It is immediate to see that all bundles appearing in the complex above satisfy this condition. Hence, $\mA_2|_X$ is acyclic.

 \subsubsection*{Step 3: cohomology of $\mA_1|_X$}
 In order to compute the cohomology of $\mA_1|_X$ we first need to tensor complex \ref{koszulLR} with $\mA_1$.

 We get the following complex, after decomposing the plethysms in irreducibles:

\begin{equation}\label{kosza1}
 0 \to \Sigma_{2,2}\mQ (-7) \to  \Sigma_{2,2}\mQ \otimes (\Sym^3 \mU)(-4) \to  \Sigma_{2,2}\mQ \otimes ((\Sym^4 \mU)(-2)  \oplus \of(-4) ) \to\end{equation}
\[\to   \Sigma_{2,2}\mQ \otimes (\Sym^3 \mU)(-1) \to \mA_1 \to \mA_1|_X \to 0.
\]

 We start from left to right. Since all of these bundles are of the form $\Sigma_{2,2}\mQ \otimes \Sigma_{\beta_1, \beta_2} \mU$, any factor with $\beta_1 \in \lbrace 6,5,2,1 \rbrace$ or $\beta_2 \in \lbrace 7,6,3,2 \rbrace$ will be acyclic.
 
 In particular, the factors $ \Sigma_{2,2}\mQ (-7) $, $\Sigma_{2,2}\mQ \otimes (\Sym^4 \mU)(-2)$, $ \Sigma_{2,2}\mQ (-1) $ are acyclic, while the other factors will all have some cohomology. 
 
 Let us start from $\Sigma_{2,2}\mQ \otimes (\Sym^3 \mU)(-4) $. We use the BBW algorithm as in \cite[4.1.4, 4.1.5, 4.1.9]{wey03}, following its notations.The number of disorders of the partition $\alpha=(2,2,0,0,7,4)$ (or the \emph{length} in the Weyl group) is 6. 
 In fact, the Weyl group of $SL_n$ is the symmetric group $\Sigma_n$, and the dotted action of a simple reflection $\sigma_i=(i, i+1)$ on $\lambda=(\lambda_1, \ldots, \lambda_n)$ is $\sigma ^\cdot(\lambda):=\sigma (\lambda+\delta)-\delta$, which translates as $(\lambda_1, \ldots, \lambda_{i+1}-1, \lambda_i+1, \ldots \lambda_n)$. The minimum number of exchanges needed to bring this partition to a non-increasing one (hence the \emph{number of disorders}) is 6. The cohomology group is given by the representation associated to the partition $\beta$, which is defined as $\beta=\sigma^\cdot (\alpha)$, for the unique $\sigma$ such that $\beta$ is non-increasing. In this case them $\beta=(3,3,3,2,2,2)$.
 
 This means that the only non-zero cohomology group of this bundle is $$H^6(\Gr(2,6),\Sigma_{2,2}\mQ \otimes (\Sym^3 \mU)(-4) ) \cong \W^3 V_6^{\vee},$$ (see e.g. \cite[Example 2.1.17]{wey03}) which is of course 20-dimensional.

 If we proceed with our analysis, the next bundle which is not acyclic is 
    $\Sigma_{2,2}\mQ(-4) )$. We immediately see by running the BBW algorithm on the corresponding partition, that we have
    $$
    H^4(\Gr(2,6), \Sigma_{2,2}\mQ(-4)) \cong \CC.
    $$
    Finally, we are left with $\Sigma_{2,2}\mQ \otimes (\Sym^3 \mU)(-1)$. Again, we can check that we have a 20-dimensional cohomology in degree 2: to be precise,
     $$
    H^2(\Gr(2,6),\Sigma_{2,2}\mQ \otimes (\Sym^3 \mU)(-1)) \cong \W^3 V_6^{\vee}.
    $$

    In order to finish the proof, we just need to compute the cohomology of $ \Sigma_{2,2}\mQ(-1)|_X $ from the sequence \ref{kosza1}. In general, to compute the cohomology on $X$ of the restriction of any vector bundle $\mathcal{B}$ on $\Gr(2,6)$ there is a spectral sequence, see \cite[\S2, 11]{bor} $$E_1^{-q,p}= H^p(\Gr(2,6), \W^q \Sym^3 \mU \otimes \mathcal{B}) \Rightarrow H^{p-q}(X, \mathcal{B}|_X).$$
    In this case, the situation is particularly simple: the couples $(p,q)$ for which there is cohomology are $(3,6),(2,4),(1,2)$: hence each $H^i(X, \mA_1|_X)$ has contribution from exactly one of these values. In particular,  $$H^1(X, \mA_1|_X) \cong H^3(X, \mA_1|_X) \cong \W^3 V_6^{\vee}, \ \ H^2(X, \mA_1|_X)\cong \CC.$$

    This concludes the proof, since $\mA_2|_X$ is acyclic, and therefore $H^p(X,End^0(\mF)) \cong H^p(X,\mA_1|_X)$.
    \end{proof}

Finally, we have now to prove that $\mF$ defined as above is slope-stable. For this, we use the following:
\begin{lemma} \label{prop:stability1} Let $X$ be generic in its polarized moduli space, $\mF$ as above. Then $\mF$ is slope-stable.  
\end{lemma}

\begin{proof}
Since $X$ is generic, we can assume that its Picard group is cyclic, generated by the restricted Pl\"ucker polarization $h$. Our $\mF$ is the second exterior power of $\mE$, which is slope-stable by \cite[8.6]{og19}. Since it is the exterior power of a slope-stable bundle, by \cite[3.2.11]{hl10}, $\mF$ is polystable, hence direct sum of stable vector bundles of the same slope, see also \cite{don85,don87, uy86}. On the other hand, from the computation above one deduce that $\mF$ is simple, since $ext^0(\mF, \mF)= 1$. Hence, $\mF$ has in fact only one summand, and is therefore slope-stable.

\end{proof}

Putting all of these together, we have the following corollary:
\begin{corollary} \label{recap1} The vector bundle $\mF$ on $X$ is slope-stable, modular and non-rigid. Its ext-table is given in Thm \ref{thm:ext}.
\end{corollary}

We can immediately compute the numerical invariants of $\mF$, for example its Chern character and total Chern class.

\begin{proposition} \label{prop:chern} Let $h=c_1(\of_X(1))$, with the latter denoting the restriction of the Pl\"ucker line bundle. The total Chern class and character of $\mF$ are given by 
\begin{equation}
    \ch(\mF)= 6 +3h +\frac{1}{4} \left( 3h^2-c_2(X) \right)-\frac{1}{20} h c_2(X) -\frac{c_4(X)}{h_1^4} \end{equation}
    \[
    c(\mF)= 1+3h+ \left(\frac{c_2(X)+15h^2}{4}\right)+2c_2(X)h + \frac{7}{12}c_4(X).
\]
\end{proposition}
\begin{proof}
We perform a standard computation using the splitting principle and the normal sequence for $X$. Notice that the Chern character up to degree 2 is determined by \ref{lem:chernwedge}. 

If we denote by $h_i:= c_i(\mQ|_X)$, (hence $h=h_1$) we first obtain the expression (of independent interest)

\[
\ch(\mF)= 6+3h+ \left( \frac{3}{2}h^2-2h_2 \right)- \frac{1}{2}h_3- \frac{1}{6} h_2h_3.
\]
On the other hand, we can express the $h_i$ in terms of the (even) Chern classes of $X$. For example, we have
\[
c_2(X)=-3h^2+8h_2; \ \ c_2(X)h=10 h_3; \ \ c_4(X)=18h_1h_3.
\]
Hence $\frac{3}{2} h^2- \frac{c_2(X)+3h^2}{4}=\frac{3h^2-c_2(X)}{4}.$ Notice also that $\frac{h_2h_3}{6}=\frac{c_4(X)}{108}$, but 108 is exactly the degree of $X$ with respect to the Pl\"ucker embedding, i.e. $h^4$, see e.g. \cite[6.2-(ii)]{huy}.
We point out that these computations can be sped up and verified using Macaulay2, \cite{m2}. We add the required input for sake of completeness.
\begin{verbatim}
G=flagBundle({2,4});
B=bundles G;
N=symmetricPower(3, dual B_0);
X=sectionZeroLocus(N);
F=exteriorPower(2,B_1 **OO_Y);
ch(F)
chern(2,tangentBundle X)
chern(4,tangentBundle X)
chern(2, tangentBundle X)*H_(2,1)
\end{verbatim}

The strategy is the same in the case of the total Chern class, starting from the expression
\[
1+3h+(3h^2+2h_2)+20 h_3+ \frac{21}{2}h_2h_3,
\]
and performing the same manipulations.
\end{proof}

\begin{remark} \label{rmk:dv}
The proof just concluded shows how to produce a non-rigid slope-stable modular vector bundle starting from a rigid one on a hyperk\"ahler. Another example of such a (rigid) vector bundle on a HK is given by the restriction of the quotient bundle $\mQ|_Z$ on the Debarre-Voisin hyperk\"ahler $Z=(\Gr(6,10), \W^3\mU^{\vee})$, see \cite[Thm 1.4, Rmk 1.5]{og19} for a proof of the rigidity. If we consider once again $\W^2 \mQ|_Z$ we get the same cohomologies as our $\mF$ on the varieties of lines $X \subset \Gr(2,6)$. The computations are analogous to the case above (but longer) and we will spare them to the reader. This fact may seems a striking coincidence, but we expect that these two examples should in fact be the same. In fact, the projectified of the (restricted) quotient bundle on $X$ deforms on the projectified of the restricted quotient bundle on $Z$. More generally, it should hold that the projectivization of slope-stable vector bundles of the same rank considered in \cite{og19} belong to a unique relative family on the moduli space of hyperk\"ahler manifolds.
    
\end{remark}
\section{Another non-rigid example.}
The example in the above section was produced with the purpose of constructing a slope-stable, non-rigid modular vector bundle with $\Ext^2(\mF, \mF)$ as small as possible as explained in the introduction, see also \cite{bot22,markman}. If one drops the last requirement, it is interesting to ask how easy is to produce examples of non-rigid, slope-stable, modular vector bundles. In fact, this seems to be the case, and Proposition \ref{lem:modularlambda} seems to be the key to produce other examples.
The first idea is to iterate the construction behind Theorem \ref{thm:ext}. In what follows, $X$ is always the Fano variety of lines on a cubic 4-fold. Using the same notations as above, consider on $X$ the rank 15 vector bundle  $\mK:= \W^2 \mF$. By definition, Proposition \ref{lem:modularlambda} applies to this case as well. It is still irreducible , and in particular we have the isomorphism
\[
\mK \cong \Sigma_{2,1,1}\mQ|_X.
\]

As before, also the bundle $\mK$ is slope-stable: in fact we have the following proposition.
\begin{lemma} \label{prop:stability2}
Let $X$ be generic in its polarized moduli space, $\mK$ as above. Then $\mK$ is slope-stable.  \end{lemma}
\begin{proof}
The proof is the same as the one in Lemma \ref{prop:stability1}.
   
\end{proof}

Let us compute the invariants of $\mK$.
\begin{proposition}

    The vector bundle $\mK$ has rank 15 and degree 15. Its total Chern class/character are given by the following expressions:

\begin{equation}
    \ch(\mK) = 15+15h+ \left( \frac{15h^2-2c_2(X)}{2} 
    \right) + \frac{c_2(X) h}{2}+ \frac{c_4(X)}{2h^4},
\end{equation}

\[
c(\mK)= 1+15h+(105h^2+c_2(X))+286 c_2(X)h+\frac{2995}{6} c_4(X).
\]
 
\end{proposition}

\begin{proof}
     As before, we denote  $h_i:=c_i(\mQ|_X)$. We first compute the Chern character and class using these invariants. 
    \[
    \ch(\mK)=15+15h_1+ \left(\frac{21}{2}h_1^2-8 h_2\right)+5 h_3 + \frac{1}{12}h_1h_3.
    \]
         \[
c(\mK)= 1+ 15h_1 + (102 h_1^2 +8 h_2)+2860 h_3+8985 hh_3
    \]
    Then we apply the same coordinate changes already used in Proposition \ref{prop:chern} to obtain the above result. As in the case of Proposition \ref{prop:chern}, one can also use \cite{m2}, with minimal changes to the code already provided.
\end{proof}

We can replicate the computations of Theorem \ref{thm:ext}, and obtain the following result:

\begin{theorem} \label{thm:ext2}
    Set $\mK =\W^2 (\W^2 \mE)$ on $X$, using the above notation. Then $\mK$ is modular, but non-rigid. In particular, if we define $End^0(\mK)$ as $(\mK \otimes \mK^{\vee}) / \of_X$, we have \[ H^p(X, End^0(\mK)) \cong
    \begin{cases}
         20 \ (p=1,3) \\
        190 \ \ (p=2) \\
        0 \ \ \textrm{otherwise}.
    \end{cases}\]

    In fact, we have the following isomorphism:
    \[
 H^*(X, End^0(\mK)) \cong H^*(\End^0 \mF) \oplus (\Sigma_{2,2,1,1} V_6^{\vee} \oplus \CC )[2]
    \]
    In particular, we have that $$\W^2 \Ext^1(\mK, \mK) \cong H^2(End^0(\mK)),$$
    since the first space is also isomorphic to
   $ 
\W^2 (\W^3 V_6^{\vee}) \cong \Sigma_{2,2,1,1} V_6^{\vee} \oplus \CC.    
    $
\end{theorem}

\begin{proof}

   The proof is similar in nature to the proof of Theorem \ref{thm:ext}. We first need to decompose of $\mK \otimes \mK^{\vee}$ in irreducibles. Such a decomposition is:
 \[
 \mK \otimes \mK^{\vee} \cong \Sigma_{4,2,2}\mQ(-2) \oplus \Sigma_{3,1}\mQ(-1) \oplus \Sigma_{3,3,2}\mQ(-2)\oplus \Sigma_{2,2}\mQ(-1) \oplus \Sigma_{2,1,1}\mQ(-1) \oplus \of
 \]

We can replicate the Borel-Bott-Weil and Koszul computation of Theorem \ref{thm:ext} and check that $\Sigma_{3,1}\mQ(-1)|_X$, $\Sigma_{3,3,2}\mQ(-2)|_X$ (and of course also $\Sigma_{2,1,1}\mQ(-1)|_X$) are acyclic.

Excluding $\of_X$, there are two other factors appearing in the decomposition above. One of them is $\Sigma_{2,2}\mQ(-1)|_X$, the one already appearing in Theorem \ref{thm:ext}. This is in turn responsibile for $\Ext^p(\mK, \mK)$ for $p=1,3$.

The other cohomology group is given by $\Sigma_{4,2,2} \mQ(-2)|_X$, where the factor $\W^2 \Sym^3 \mU \otimes \Sigma_{4,2,2} \mQ(-2) \cong (\Sym^4 \mU (-1) \oplus \of(-3)) \otimes \Sigma_{4,2,2} \mQ(-2)$ contains $\Sigma_{4,2,2} \mQ \otimes \Sym^4 \mU (-3)$ that cohomology in degree 4 isomorphic to $\Sigma_{2,2,1,1}V_6^{\vee}$. This implies that $H^2(X, \Sigma_{4,2,2} \mQ(-2)|_X) \cong H^4(\Gr(2,6), \Sigma_{4,2,2} \mQ \otimes \Sym^4 \mU (-3) ) \cong \Sigma_{2,2,1,1}V_6^{\vee}$.
 
In particular, it follows that $\Ext^1(\mK, \mK) \cong \W^3 V_6^{\vee}$, whose second exterior power is $ \Sigma_{2,2,1,1} V_6^{\vee} \oplus \CC$. 
On the other hand, we have that $\Ext^2(End^0(\mK)) \cong H^2(X, \Sigma_{4,2,2} \mQ(-2)|_X) \oplus H^2(X, \Sigma_{2,2}\mQ(-1)|_X) \cong \Sigma_{2,2,1,1} V_6^{\vee} \oplus \CC$. Therefore, we get
\[
\W^2 \Ext^1(\mK, \mK) \cong \Ext^2(End^0(\mK)),
\]
which concludes the proof.

\end{proof}

The above example is quite interesting. In fact, since $\mK$ is projectively hyperholomorphic, then \cite{mo23} implies that $\mK$ has a formal DG Lie algebra. Moreover, the isomorphism $\W^2 \Ext^1(\mK, \mK) \cong H^2(End^0(\mK)),$ implies that the Yoneda product is skew-symmetric and therefore that the Kuranishi space is smooth.

\begin{remark}
    As in the case of $\mF$, we can consider the restriction of the same Schur functor on the Debarre-Voisin example, and obtain the same results. However, the same reasoning of remark \ref{rmk:dv} also applies here. 
\end{remark}

\begin{remark}
One might wonder what happens if we keep on raising $\mF$ or $\mK$ to the second or higher wedge powers. Unsurprisingly, they quickly becomes reducible: this is for example the case of $\W^2\W^2 \W^2 \mQ|_X=\W^2\mK$, that decomposes as $\Sigma_{4,2,1,1}\mQ|_X \oplus \Sigma_{3,3,2}\mQ|_X \oplus \Sigma_{3,2,2,1}\mQ|_X$. Of course, one can consider a single irreducible factor and try to replicate the construction: for example, considering the ext-cohomology $\Ext^{\bullet}(\Sigma_{4,2,1,1}\mQ|_X,\Sigma_{4,2,1,1}\mQ|_X)$ we get always a 20-dimensional $ext^1$, and a quite large $ext^2$. The same happens for $\Sym^2 \W^2 \mE$, and other cases. For example, the vector bundle $\Sigma_{2,2}\mQ|_X$, which appears in the decomposition of $\Sym^2 (\Sym^2 \mQ|_X)$, has 20-dimensional $ext^1$ and 590-dimensional $ext^2$. Of course the modularity Proposition \ref{lem:modularlambda} does not apply anymore, and one should come up with a similar formula, but we do not necessarily see the case for it. One could argue that this is in fact a general phenomenon, with these bundle obtained from linear algebra which are in general non-rigid, with a higher dimensional $ext^2$ everytime.
 
\end{remark}

\bibliographystyle{alpha}

\end{document}